\documentclass[12pt,a4paper]{article}

\usepackage{latexsym,amssymb,amsmath}

\begin{document}

\newtheorem{define}{Definition}
\newtheorem{proposition}[define]{Proposition}
\newtheorem{theorem}[define]{Theorem}
\newtheorem{lemma}[define]{Lemma}
\newtheorem{corollary}[define]{Corollary}
\newtheorem{problem}[define]{Problem}
\newtheorem{conjecture}[define]{Conjecture}
%
%
\newenvironment{proof}{
\par
\noindent {\bf Proof.}\rm}%
{\mbox{}\hfill\rule{0.5em}{0.809em}\par}

\baselineskip=22pt
\parindent=0.6cm


\title{A Common Generalization of the Theorems of 
Erd\H{o}s-Ko-Rado and Hilton-Milner}

\author{
Wei-Tian Li\\
\normalsize Department of Applied Mathematics \\
\normalsize National Chung Hsing University \\
\normalsize Taichung 40227, Taiwan\\
\normalsize {\tt Email:weitianli@dragon.nchu.edu.tw}\\
\and
Bor-Liang Chen\\
\normalsize Department of Business Administration\\
\normalsize National Taichung University of Science and Technology\\
\normalsize Taichung 40401, Taiwan\\
\normalsize {\tt Email:blchen@ntit.edu.tw}
\and
Kuo-Ching Huang\\
\normalsize Department of Financial and Computational Mathematics\\
\normalsize Providence University\\
\normalsize Taichung 43301, Taiwan\\
\normalsize {\tt Email:kchuang@gm.pu.edu.tw}
\and
Ko-Wei Lih \thanks{Research supported by NSC 
(No. 100-2517-S-001-001-MY3)} \\
\normalsize Institute of Mathematics\\
\normalsize Academia Sinica\\
\normalsize Taipei 10617, Taiwan\\
\normalsize {\tt Email:makwlih@sinica.edu.tw}
}

\date{\small }

\maketitle

\newpage

\begin{abstract}
\noindent
Let $m, n$, and $k$ be integers satisfying $0 < k \leqslant n < 2k 
\leqslant m$. A family of sets $\mathcal{F}$ is called an 
$(m,n,k)$-{\em intersecting family} if $\binom{[n]}{k} \subseteq 
\mathcal{F} \subseteq \binom{[m]}{k}$ and any pair of members of 
$\mathcal{F}$ have nonempty intersection. Maximum $(m,k,k)$- and 
$(m,k+1,k)$-intersecting families are determined by the theorems 
of Erd\H{o}s-Ko-Rado and Hilton-Milner, respectively. We determine 
the maximum families for the cases $n = 2k-1, 2k-2, 2k-3$, and $m$ 
sufficiently large.


\end{abstract}


%
\section{Introduction}
%

For positive integers $a \leqslant b$, define $[a,b]=\{a, a+1, 
\ldots , b\}$ and let $[a]=[1, a]$. The cardinality of a set 
$X$ is denoted by $|X|$. A set of cardinality $n$ is called an 
$n$-set. A family of subsets of $X$ is said to be {\em intersecting} 
if no two members are disjoint. The family of all $k$-subsets of 
$X$ is denoted by $\binom{X}{k}$. Note that $\binom{[m]}{k}$ is 
intersecting if $0 < k \leqslant m < 2k$. If all members of a family 
$\mathcal{F} \subseteq \binom{[m]}{k}$ contain a fixed element, 
then $\mathcal{F}$ is obviously an intersecting family and is said 
to be {\em trivial}. A trivial intersecting family can have at most 
$\binom{m-1}{k-1}$ members. One of the cornerstones of the extremal 
theory of finite sets is the following pioneering result of Erd\H{o}s, 
Ko, and Rado \cite{EKR}.

\begin{theorem}\label{EKR}
Suppose $0 < 2k < m$. Let $\mathcal{F} \subseteq \binom{[m]}{k}$ 
be an intersecting family. Then $|\mathcal{F}| \leqslant 
\binom{m-1}{k-1}$. Moreover, the equality holds if and only 
if $\mathcal{F}$ consists of all $k$-subsets containing a fixed
element.
\end{theorem}

Let $A \in \binom{[m]}{k}$ and $t \not\in A$. Define $\mathcal{M}_1(A;t)
= \{A\} \cup \{B \in \binom{[m]}{k} \mid t \in B \mbox{ and } A 
\cap B \neq \emptyset \}$. Clearly $|\mathcal{M}_1(A;t)| = \binom{m-1}{k-1} 
- \binom{m-1-k}{k-1} + 1$. Let $X \in \binom{[m]}{3}$. Define
$\mathcal{M}_2(X) = \{ B \in \binom{[m]}{k} \mid |X \cap B| \geqslant 2\}$.
Both $\mathcal{M}_1(A;t)$ and $\mathcal{M}_2(X)$ are intersecting families.
The largest size of a non-trivial intersecting family was determined in 
the following result of Hilton and Milner \cite{HM}.

\begin{theorem}\label{HM}
Suppose $0 < 2k < m$. Let $\mathcal{F} \subseteq \binom{[m]}{k}$ be an 
intersecting family such that $\cap \{A \mid A \in \mathcal{F}\} = \emptyset$. 
Then $|\mathcal{F}| \leqslant \binom{m-1}{k-1} - \binom{m-1-k}{k-1} + 1$.
Moreover, the equality holds if and only if $\mathcal{F}$ is of the form
$\mathcal{M}_1(A;t)$ or the form $\mathcal{M}_2(X)$, the latter occurs only 
for $k = 3$.
\end{theorem}  

In a more general form, the Erd\H{o}-Ko-Rado theorem describes the size and 
structure of the largest collection of $k$-subsets of an $n$-set having the 
property that the intersection of any two subsets contains at least $t$ elements.
The Erd\H{o}-Ko-Rado theorem has motivated a great deal of development of finite extremal set 
theory since its first publication in 1961. The complete establishment of 
the general form was achieved through cumulative works of Frankl \cite{FRANKL78},
Wilson \cite{WILSON}, and Ahlswede and Khachatrian \cite{AK97}. Ahlswede and 
Khachatrian \cite{AK96} even extended the Hilton-Milner theorem in the general 
case. The reader is referred to Deza and Frankl \cite{DF}, Frankl \cite{FRANKL87}, 
and Borg \cite{Borg} for surveys on relevant results.

Let $0 < k \leqslant n < 2k \leqslant m$. We call an intersecting family 
$\mathcal{F}$ an $(m,n,k)$-{\em intersecting family} if $\binom{[n]}{k} 
\subseteq \mathcal{F} \subseteq \binom{[m]}{k}$. Define $\alpha(m,n,k)=
\max\{|\mathcal{F}| \mid \mathcal{F} \mbox{ is an $(m,n,k)$-intersecting 
family}\}$. An $(m,n,k)$-intersecting family with cardinality $\alpha(m,n,k)$ 
is called a {\em maximum} family. The focus for our study is the following.

\begin{problem}\label{prob_1}
For $0 < k \leqslant n < 2k \leqslant m$, determine $\alpha(m,n,k)$ 
and the corresponding maximum families.
\end{problem}

Suppose that $\mathcal{F}$ is an $(m,n,k)$-intersecting family. 
If any $A \in \mathcal{F}$ satisfies $|A \cap [n]| \leqslant n - k$,
then $|[n] \setminus A| \geqslant n-(n-k) = k$. Hence, there exists 
a $k$-subset $B \subseteq[n] \setminus A$. It is clear that $B \in 
\mathcal{F}$ and $B\cap A = \emptyset$, violating the intersecting
condition on $\mathcal{F}$. Hence, we have a size constraint on any 
$A \in \mathcal{F}$: $|A \cap [n]| \geqslant n-k+1$, or equivalently,
$|A \setminus [n]| \leqslant 2k-n-1$.

For any fixed $t \in [n]$, define 
\[ 
\mathcal{H}_t^{m,n,k} = \binom{[n]}{k} \cup \bigcup_{i=1}^{2k-n-1} 
\left \{ A \cup B \cup \{t\} \left|  A\in \binom{[n] \setminus \{t\}} 
{k-i-1}, B\in \binom{[n+1,m]}{i} \right. \right \}.
\]
We often write $\mathcal{H}_t$ for $\mathcal{H}_t^{m,n,k}$ if the context
is clear. It is easy to see that $\mathcal{H}_t$ is an $(m,n,k)$-intersecting 
family and its cardinality is equal to
\[
h(m,n,k) = \binom{n}{k} +
\sum_{i=1}^{2k-n-1} \binom{n-1}{k-i-1}\binom{m-n}{i}.
\]
Hence, $\alpha(m,n,k) \geqslant h(m,n,k)$. 

For the case $n = k$, Theorem \ref{EKR} shows that $\alpha(m,k,k) 
= \binom{m-1}{k-1} = h(m,n,k)$ and all maximum families are of the 
form $\mathcal{H}_t$ for some $t \in [k]$. For the case $n = k+1$, 
a maximum family is non-trivial since $\binom{[k+1]}{k} = \{[k+1] 
\setminus \{i\} \mid 1\leqslant i\leqslant k+1\}$ and $\cap \{A \mid 
A \in \binom{[k+1]}{k}\} = \emptyset$. Theorem \ref{HM} shows that 
$\alpha(m,k+1,k) = \binom{m-1}{k-1} - \binom{m-1-k}{k-1} + 1 = 
h(m,k+1,k)$ and all maximum families are of the form $\mathcal{M}_1(A;t) 
= \mathcal{H}_t$, where $t \in [k+1]$ and $A = [k+1] \setminus \{t\}$,
or the form $\mathcal{M}_2(X)$, where $X \in \binom{[4]}{3}$, the 
latter occurs only for $k = 3$.

In view of the above paragraph, a solution of Problem \ref{prob_1} can be 
regarded as a common generalization of both the Erd\H{o}s-Ko-Rado and the 
Hilton-Milner theorems. For these two particular cases, the obvious lower
bound $h(m,n,k)$ coincides with the maximum value and, except the case for 
$k = 3$ and $n = 4$, all maximum families are of the form $\mathcal{H}_t$. 
This phenomenon leads us to pose the following.

\begin{problem}\label{prob_2}
When does $\alpha(m,n,k) = h(m,n,k)$ hold? When it does, are 
$\mathcal{H}_t$'s the only maximum families?
\end{problem}

In this paper, we answer the above questions for the cases $n = 2k-1, 2k-2, 
2k-3$, and $m$ sufficiently large.

%
\section{The cases for $m = 2k$ and $n = 2k-1$}
%

\begin{proposition}\label{2KK}
We have $\alpha(2k,n,k) = \frac{1}{2}\binom{2k}{k} = h(2k,n,k)$.
\end{proposition}

This is true because any $(2k,n,k)$-intersecting family cannot
contain a $k$-subset and its complement in $[2k]$ simultaneously.
Any maximum family $\mathcal{F}$ can be obtained in the following 
manner. Pick a pair of a $k$-subset $A$ and its complement $A'=[2k] 
\setminus A$. If $A$ or $A'$ is a subset of $[n]$, then we put it 
in $\mathcal{F}$. Otherwise, we put any one of them in $\mathcal{F}$.

A special case of the above construction for a maximum family is
to choose the one that contains a prescribed element $t$ when neither
$A$ nor $A'$ is a subset of $[n]$. If $t\in[n]$, then the family so
constructed is precisely $\mathcal{H}_t$. 

\medskip
\noindent
{\bf Convention}.\
From now on, we always assume that $0 < k \leqslant n < 2k < m$ for 
any $(m,n,k)$-intersecting family.

\begin{proposition}
For $n = 2k-1$, we have $\alpha(m,n,k) = \binom{n}{k} = h(m,n,k)$ and 
$\binom{[n]}{k}$ is the unique maximum $(m,n,k)$-intersecting family.
\end{proposition}

\begin{proof}
Let $\mathcal{F}$ be a maximum $(m,n,k)$-intersecting family. 
For any $A \in \mathcal{F}$, we know $k \geqslant |A \cap[n]| 
\geqslant n-k+1=k$. Thus, $A \in \binom{[n]}{k}$, and 
hence $\mathcal{F} \subseteq \binom{[n]}{k}$. Therefore, 
$\mathcal{F} = \binom{[n]}{k}$ and $\alpha(m,n,k) = 
|\mathcal{F}| = \binom{n}{k} = h(m,n,k)$. 
\end{proof}

%
\section{The case for $n = 2k-2$}
%

Frequently, extremal problems concerning sub-families of
$\binom{[m]}{k}$ can be translated into the context of Kneser
graphs so that graph-theoretical tools may be employed to solve
them. For $0 < 2k \leqslant n$, a {\em Kneser graph} ${\sf KG}(n,k)$
has vertex set $\binom{[n]}{k}$ such that two vertices $A$ and
$B$ are adjacent if and only if they are disjoint as subsets.
By stipulation, we use ${\sf KG}(n,k)$ to denote the graph consisting 
of $\binom{n}{k}$ isolated vertices when $0 < k \leqslant n < 2k$.
An {\em independent} set in a graph is a set of vertices no two of
which are adjacent. The maximum cardinality of an independent set in a 
graph $G$ is called the {\em independence number} of $G$ and is denoted 
by $\alpha(G)$. The Erd\H{o}s-Ko-Rado theorem just gives the independence
number of a Kneser graph and characterizes all maximum independent sets.

The {\em direct product} $G \times H$ of two graphs $G$ and 
$H$ is defined on the vertex set $\{(u,v) \mid u \in G \mbox{ and } 
v \in H\}$ such that two vertices $(u_1,v_1)$ and $(u_2,v_2)$ are 
adjacent if and only if $u_1$ is adjacent to $u_2$ in $G$ and $v_1$ is 
adjacent to $v_2$ in $H$. The cardinality of the vertex set of a graph
$G$ is denoted by $|G|$. The following result is due to Zhang \cite{Zhang}.

\begin{theorem}\label{DIR}
Let $G$ and $H$ be vertex-transitive graphs. Then $\alpha(G \times H) = 
\max \{\alpha(G)|H|,|G|\alpha(H)\}$. Furthermore, every maximum independent 
set of $G \times H$ is the pre-image of an independent set of $G$ or $H$ 
under projection.
\end{theorem}

Since Kneser graphs are vertex-transitive, we are going to use the
above theorem for $G = {\sf KG}(n_1,k_1)$ and $H = {\sf KG}(n_2,k_2)$.
The version of Theorem \ref{DIR} for Kneser graphs was established in 
an earlier paper \cite{FRANKL} of Frankl.

Suppose that $\mathcal{F}$ is an $(m,n,k)$-intersecting family.
Define its {\em canonical partition} as follows.
\[
\mathcal{F} = \binom{[n]}{k} \cup (\bigcup_{i=1}^{2k-n-1}
\mathcal{F}_i),
\]
where $\mathcal{F}_i = \{F \in \mathcal{F} \mid |F \cap [n]| = k-i
\mbox{ and } |F \cap [n+1,m]| = i\}$. For each $i$, we define an 
injection $f_i$ from $\mathcal{F}_i$ to the vertex set of ${\sf KG}(n,k-i) 
\times {\sf KG}(m-n,i)$ such that $f_i(F) = (A,B^*)$, where $A=F \cap [n]$ 
and $B^* = \{b-n \mid b \in F \mbox{ and } b \geqslant n+1 \}$. Since 
$\mathcal{F}_i$ is intersecting, it is easy to verify that the image of 
$f_i$ is an independent set of ${\sf KG}(n,k-i) \times {\sf KG}(m-n,i)$. 
Thus, $|\mathcal{F}_i| \leqslant \alpha({\sf KG}(n,k-i) \times {\sf KG}(m-n,i))$.
We immediately obtain the following upper bound.
\[
|\mathcal{F}| \leqslant \binom{n}{k} + \sum_{i=1}^{2k-n-1}
\alpha({\sf KG}(n,k-i) \times {\sf KG}(m-n,i)).
\]

We can derive the following by Theorem \ref{EKR}, Theorem \ref{DIR}, 
and direct computation.

\begin{lemma}\label{lem_max}
When $2(k-i) \leqslant n$ and $2i \leqslant m-n$,  
\[
\alpha({\sf KG}(n,k-i) \times {\sf KG}(m-n,i)) = \left \{
\begin{array}{ll}
\binom{n-1}{k-i-1}\binom{m-n}{i} & \mbox{ if } m \geqslant nk/(k-i),
\vspace{0.3cm}\\
\binom{n}{k-i}\binom{m-n-1}{i-1} & \mbox{ otherwise}.
\end{array}
\right.
\]
When $2(k-i) > n$ or $2i > m-n$, $\alpha({\sf KG}(n,k-i) \times 
{\sf KG}(m-n,i)) = \binom{n}{k-i}\binom{m-n}{i}$.
\end{lemma}

\begin{theorem}
For $n = 2k - 2$, we have $\alpha(m,n,k) = h(m,n,k)$. All the maximum
families are of the form $\binom{[2k-2]}{k}\cup\{F\cup\{b\}\mid F\in
\mathcal{F}^*,b\in[2k-1,m]\}$, where $\mathcal{F}^*$ is any maximum 
intersecting family of $(k-1)$-subsets of $[2k-2]$.
\end{theorem}

\begin{proof}
Let $\mathcal{F}$ be a largest $(m,2k-2,k)$-intersecting family with canonical
partition $\binom{[2k-2]}{k} \cup \mathcal{F}_1$. Now, all the conditions
$2(k-1) \leqslant n$, $2 \leqslant m-n$, and $m \geqslant nk/(k-1)$ hold.
It follows from Lemma \ref{lem_max} that $|\mathcal{F}_1| \leqslant \binom{2k-3} 
{k-2} \binom{m-2k+2}{1}$. Then $|\mathcal{F}| = \binom{2k-2}{k} +
|\mathcal{F}_1| \leqslant h(m,2k-2,k)$. As a consequence,  $|\mathcal{F}| = 
h(m,2k-2,k)$ and $|\mathcal{F}_1| = \binom{2k-3}{k-2}\binom{m-2k+2}{1}$. 
By Theorem \ref{DIR}, $f_1(\mathcal{F}_1)$ is a maximum independent set in 
${\sf KG}(2k-2,k-1) \times {\sf KG}(m-2k+2,1)$ and the collection 
$\mathcal{F}^*$ of all the first components of $f_1(\mathcal{F}_1)$ is an 
independent set of ${\sf KG}(2k-2,k-1)$. Clearly, $\mathcal{F}^*$ is maximum 
because of its cardinality. 
\end{proof}

\bigskip

\noindent
{\bf Remark.} When $k=3$, an $(m,2k-2,k)$-family is also an $(m,k+1,k)$ family. 
There are other maximum families besides the collection of all $\mathcal{H}_t$'s. 
This phenomenon is consistent with the Hilton-Milner theorem for the case $k = 3$.

%
\section{The case for $n=2k-3$}
%

Two families of sets $\mathcal{A}$ and $\mathcal{B}$ are said to be 
{\em cross-intersecting} if $A \cap B \neq \emptyset$ for any pair 
$A \in \mathcal{A}$ and $B \in \mathcal{B}$. Frankl and Tokushige
\cite{FRATOK} proved the following.

\begin{theorem}\label{XFM}
Let $\mathcal{A} \subseteq \binom{X}{a}$ and $\mathcal{B} \subseteq 
\binom{X}{b}$ be nonempty cross-intersecting families of subsets of $X$. 
Suppose that $|X| \geqslant a+b$ and $a \leqslant b$. Then
\[
|\mathcal{A}| + |\mathcal{B}| \leqslant \binom{|X|}{b} - \binom{|X| - a}{b}+1.
\]
\end{theorem}

The above inequality provides a useful tool for handling our problems.

\begin{theorem}\label{2K-3}
For $n = 2k-3$, we have $\alpha(m,n,k) = h(m,n,k)$. All the maximum
families are of the form $\mathcal{H}_t$ for some $t \in [n]$.
\end{theorem}

\begin{proof}
Let $\mathcal{F}$ be a largest $(m,2k-3,k)$-intersecting family with
canonical partition $\binom{[2k-3]}{k} \cup \mathcal{F}_1 \cup \mathcal{F}_2$.
We further partition $\mathcal{F}_1$ and $\mathcal{F}_2$ into subfamilies.
Let $N=\binom{2k-3}{k-1}$. Partition $\binom{[2k-3]}{k-1}$ into $A_1, 
\ldots , A_N$ and $\binom{[2k-3]}{k-2}$ into $A_1', \ldots , A_N'$ such that 
$A_j \cup A_j' = [2k-3]$ for all $j$. Define $\mathcal{F}(A_j) = \{F \in 
\mathcal{F} \mid F \cap [2k-3] = A_j \}$ and $\mathcal{F}(A_j') = \{F \in 
\mathcal{F} \mid F \cap [2k-3] = A_j'\}$. 
Then 
\[
\mathcal{F} = \binom{[2k-3]}{k} \cup \Bigg( \bigcup_{j=1}^N \Big( \mathcal{F}(A_j) 
\cup \mathcal{F}(A_j') \Big) \Bigg).
\]

\bigskip

\noindent 
{\bf Observation}. If $\mathcal{F}(A_j) \neq \emptyset$, then
$|\mathcal{F}(A_j)| + |\mathcal{F}(A_j')| \leqslant m-2k+3$.

\bigskip

If $\mathcal{F}(A_j')=\emptyset$, then $|\mathcal{F}(A_j)| + 
|\mathcal{F}(A_j')| = |\mathcal{F}(A_j)| \leqslant |\{A_j \cup \{b\} 
\mid b \in [2k-2,m] \}| = m-2k+3$. If $\mathcal{F}(A_j') \neq \emptyset$, 
then $\{\{b\} \mid A_j \cup \{b\} \in \mathcal{F}(A_j)\} \subseteq 
\binom{[2k-2,m]}{1}$ and $\{\{b_1,b_2\} \mid A_j' \cup \{b_1,b_2\} \in
\mathcal{F}(A_j')\} \subseteq \binom{[2k-2,m]}{2}$ are cross-intersecting.
By Theorem \ref{XFM}, $|\mathcal{F}(A_j)| + |\mathcal{F}(A_j')| \leqslant
\binom{m-2k+3}{2} - \binom{m-2k+2}{2} + 1 = m-2k+3$. Hence, the observation
holds.

\bigskip

Now suppose that all of $\mathcal{F}(A_1), \ldots , \mathcal{F}(A_s)$ are
nonempty, yet $\mathcal{F}(A_{s+1}) = \cdots = \mathcal{F}(A_N) = \emptyset$.
Then we have
\begin{equation}\label{EQN:UPB}
|\mathcal{F}|\leqslant \binom{2k-3}{k}+s(m-2k+3)+(N-s)\binom{m-2k+3}{2}.
\end{equation}

{\em Case 1}.\ $m \geqslant 2k+2$.

Since $h(m,2k-3,k) \leqslant |\mathcal{F}|$ and $N = \binom{2k-4}{k-2} + 
\binom{2k-4}{k-3}$, it follows $s \leqslant \binom{2k-4}{k-2}$. We may 
assume $k \geqslant 5$ because $\alpha(m,3,3)$ and $\alpha(m,5,4)$ are known. 
It follows that $m \geqslant (2k-3)k/(k-2)$. Together with $2(k-2) < 2k-3$ and 
$4 < m-2k+3$, we have $\alpha({\sf KG}(2k-3,k-2) \times {\sf KG}(m-2k+3,2))
= \binom{2k-4}{k-3}\binom{m-2k+3}{2}$ by Lemma \ref{lem_max}. Recall that
$f_2(\mathcal{F}_2)$ is an independent set of ${\sf KG}(2k-3,k-2) \times 
{\sf KG}(m-2k+3,2)$. Hence, $|\mathcal{F}_2| \leqslant \binom{2k-4}{k-3}
\binom{m-2k+3}{2}$. If $s < \binom{2k-4}{k-2}$, then $|\mathcal{F}_1| =
\sum_{j=1}^{s}|\mathcal{F}(A_j)| < \binom{2k-4}{k-2}(m-2k+3)$. This leads 
to $|\mathcal{F}| = \binom{2k-3}{k} + |\mathcal{F}_1| + |\mathcal{F}_2| <
h(m,2k-3,k)$, a contradiction. Thus, $s = \binom{2k-4}{k-2}$ and 
$\alpha(m,2k-3,k) = h(m,2k-3,k)$ for $m \geqslant 2k+2$.

{\em Case 2}.\ $m = 2k+1$.

Suppose That $\binom{2k-3}{k} + \binom{2k-4}{k-2}\binom{4}{1} +
\binom{2k-4}{k-3}\binom{4}{2} = h(2k+1,2k-3,k) < |\mathcal{F}|$. 
Since $N = \binom{2k-4}{k-2} + \binom{2k-4}{k-3}$, it follows from inequality 
(\ref{EQN:UPB}) that $|\{j \mid |\mathcal{F}(A_j)| + |\mathcal{F}(A_j')|
\geqslant 5\}| > \binom{2k-4}{k-3}$. By our Observation, $|\mathcal{F}(A_j)| 
+ |\mathcal{F}(A_j')| \geqslant 5$ implies $\mathcal{F}(A_j) = \emptyset$ 
for any $j$. Thus $|\{A_j' \mid |\mathcal{F}(A_j')| \geqslant 5\}| > \binom{2k-4} 
{k-3}$. By Theorem \ref{EKR}, there exist disjoint sets $A_{j_1}'$ 
and $A_{j_2}'$ in $\{A_j' \mid |\mathcal{F}(A_j')| \geqslant 5\} \subseteq
\binom{[2k-3]}{k-2}$. Then it is easy to find two disjoint sets, one in
$\mathcal{F}(A_{j_1}')$ and the other in $\mathcal{F}(A_{j_2}')$. This 
contradicts the assumption that $\mathcal{F}$ is intersecting. Therefore
$|\mathcal{F}| = h(2k+1,2k-3,k)$.

Let us examine the maximum families. Note that $\alpha(m,2k-3,k) = h(m,2k-3,k)$
implies that inequality (\ref{EQN:UPB}) becomes equality, $s=\binom{2k-4}{k-2}$,
and $N-s=\binom{2k-4}{k-3}$. It follows that $\mathcal{F}(A_j') =
\{A_j' \cup B \mid B \in \binom{[2k-2,m]}{2}\}$ for $s < j \leqslant N$.
Since there exist non-intersecting pair $B_1$ and $B_2$ in $\binom{[2k-2,m]}{2}$,
$\{A_j' \mid s+1 \leqslant j \leqslant N\}$ must be a maximum intersecting 
family in view of its cardinality. Let $t \in \cap_{j=s+1}^N A_j'$ by Theorem
\ref{EKR}. For $1\leqslant j\leqslant s$, if there exists
$\mathcal{F}(A_{j_1}') \neq \emptyset$ for some $1 \leqslant j_1 \leqslant s$,
then there exists some $A_{j_2}'$, $s+1 \leqslant j_2 \leqslant N$ such that
$A_{j_1}' \cap A_{j_2}' = \emptyset$. We can find two disjoint sets, one in
$\mathcal{F}(A_{j_1}')$ and the other in $\mathcal{F}(A_{j_2}')$, a contradiction.
Therefore we have $\mathcal{F}(A_j') = \emptyset$ and $\mathcal{F}(A_j) = 
\{A_j \cup B \mid B \in \binom{[2k-2,m]}{1}\}$ for $1 \leqslant j \leqslant s$.
Suppose that $t \not \in A_{j_0}$ for some $1 \leqslant j_0 \leqslant s$.
Then $t \in A_{j_0}'$. For any $A_j'$, $s+1 \leqslant j \leqslant N$, we have
$A_{j_0}' \neq A_j'$ since $\mathcal{F}(A_j) = \emptyset$, yet 
$\mathcal{F}(A_{j_0}) \neq \emptyset$. Then $\{A_{j_0}', A_{s+1}, \ldots , 
A_N\}$ is an intersecting family in $\binom{[2k-3]}{k-2}$ having more than 
$\binom{2k-4}{k-3}$ members, a contradiction. Hence $\mathcal{F}$ has the form
$\mathcal{H}_t$ for $t \in [2k-3]$.
\end{proof}

%
\section{The case for $m$ sufficiently large}
%

We have solved Problem \ref{prob_1} for $n = 2k-1, 2k-2$, and $2k-3$.
In this section, we are going to assume that $k \leqslant n < 2k-3$ and
solve the problem when $m$ is sufficiently large.

Let $r, l, n$ be positive integers satisfying $r < l \leqslant n/2$, 
and let $X_1$ and $X_2$ be disjoint $n$-sets. Wang and Zhang \cite{WANZHA}
characterized the maximum intersecting families $\mathcal{F} \subseteq 
\{ F \in \binom{X_1 \cup X_2}{r+l} \mid$ $ |F \cap X_1| = r \mbox{ or } 
l \}$ of maximum cardinality. We consider a similar extremal problem.

\begin{problem}\label{prob_3}
Given integers $m, n, k, c, d$ satisfying $n < m$, $k \leqslant n < 2k-3$,
$d < c < k$, and $c + d = n$, characterize the intersecting families
$\mathcal{F} \subseteq \{F \in \binom{[m]}{k} \mid |F \cap [n]| = c
\mbox{ or } d\}$ of maximum cardinality.
\end{problem}

We can derive an asymptotic solution of the above problem as follows.

\begin{lemma}\label{LEM:FUF}
For given $n, k, c, d$ satisfying conditions in the above problem,
if $m$ is sufficiently large, then a maximum intersecting family 
$\mathcal{F}$ has the form $\{A \cup B \cup \{t\} \mid A \in \binom{[n] 
\setminus \{t\}}{c-1}, B \in \binom{[n+1,m]}{k-c}\} \cup 
\{A \cup B \cup \{t\} \mid A \in \binom{[n] \setminus \{t\}}{d-1},
B \in \binom{[n+1,m]}{k-d}\}$ for some $t \in[n]$, and hence 
$|\mathcal{F}| = \binom{n-1}{c-1}\binom{m-n}{k-c} + \binom{n-1}{d-1} 
\binom{m-n}{k-d}$.
\end{lemma}

\begin{proof}
Let $\mathcal{F}$ be a maximum intersecting family. Any special form 
stated in the lemma is an intersecting family, hence its cardinality
$\binom{n-1}{c-1}\binom{m-n}{k-c} + \binom{n-1}{d-1} \binom{m-n}{k-d}$
supplies a lower bound for $|\mathcal{F}|$.

Let us consider upper bounds for $|\mathcal{F}|$. First partition 
$\mathcal{F}$ into two subfamilies $\mathcal{F}_{k-c}$ and $\mathcal{F}_{k-d}$
such that $\mathcal{F}_{k-c} = \{F \in \mathcal{F} \mid |F \cap [n]| = c\}$
and $\mathcal{F}_{k-d} = \{F \in \mathcal{F} \mid |F\cap [n]| =d\}$.
For $\mathcal{F}_{k-d}$, we consider the injection from $\mathcal{F}_{k-d}$
to the vertex set of ${\sf KG}(n,d) \times {\sf KG}(m-n,k-d)$ defined prior
to Lemma \ref{lem_max}. We may choose $m$ sufficiently
large so that $2(k-d) < m-n$ and $m > nk/d$ hold. By Lemma \ref{lem_max}, 
we have $|\mathcal{F}_{k-d}| \leqslant \alpha({\sf KG}(n,d)) |{\sf KG}(m-n,k-d)|
= \binom{n-1}{d-1}\binom{m-n}{k-d}$. Consider a further partition on
$\mathcal{F}_{k-c}$ and $\mathcal{F}_{k-d}$. Denote $N = \binom{n}{c}$. For $A_j
\in \binom{[n]}{c}$ and $A_j' = [n] \setminus A_j$, $1\leqslant j\leqslant N$, 
let $\mathcal{F}(A_j) = \{F \in \mathcal{F}_{k-c} \mid F \cap[n] = A_j\}$ and
$\mathcal{F}(A_j') = \{F \in \mathcal{F}_{k-d} \mid F \cap [n] = A_j'\}$.
Since $A_j \cap A_j' = \emptyset$, the two families $\{B \in \binom{[n+1,m]}{k-c}
\mid A_j \cup B \in \mathcal{F}\}$ and $\{B \in \binom{[n+1,m]}{k-d} \mid A_j'
\cup B \in \mathcal{F}\}$ are cross-intersecting of size $|\mathcal{F}(A_j)|$ 
and $|\mathcal{F}(A_j')|$, respectively. Let $r \leqslant s$ be integers such 
that $\mathcal{F}(A_j) = \emptyset$ for $1 \leqslant j \leqslant r$,
$\mathcal{F}(A_j)$ and $\mathcal{F}(A_j')$ are nonempty for $r+1 \leqslant j
\leqslant s$ and $\mathcal{F}(A_j') = \emptyset$ for $s+1 \leqslant j \leqslant
N$. Then
\begin{align*}
|\mathcal{F}| &=  \sum_{j=1}^r|\mathcal{F}(A_j')| +
                    \sum_{j=r+1}^s(|\mathcal{F}(A_j)| + |\mathcal{F}(A_j')|) +
                    \sum_{j=s+1}^N |\mathcal{F}(A_j)|\\
              &\leqslant  r\binom{m-n}{k-d} + (s-r)\left( \binom{m-n}{ k-d} -
                    \binom{m-k-d}{k-d} + 1 \right) \\
              &    +(N-s)\binom{m-n}{k-c}.
\end{align*}
We first show that $r=\binom{n-1}{d-1}$. If $r > \binom{n-1}{d-1}$, then
\begin{align*}
|\mathcal{F}| &=  \sum_{j=r+1}^N|\mathcal{F}(A_j)| + \sum_{j=1}^s
                    |\mathcal{F}(A_j')| \\
              &<  (N-r)\binom{m-n}{k-c} + |\mathcal{F}_{k-d}| \\
              &\leqslant  \binom{n-1}{c-1}\binom{m-n}{k-c} + \binom{n-1}{d-1}
                    \binom{m-n}{k-d},
\end{align*} 
which cannot be true. For $m$ sufficient large, say $m > 2n(n/2)^{k-d}\binom{n} 
{\lfloor n/2 \rfloor}$, we have
\begin{align*}
&    (s-r) \left( \binom{m-n}{k-d} - \binom{m-k-d}{k-d} + 1 \right) +
      (N-s)\binom{m-n}{k-c} \\
&<  (s-r) \left( \frac{m^{k-d}}{(k-d)!} - \frac{(m-2n)^{k-d}}{(k-d)!} + 1
      \right) + (N-s)m^{k-c} \\
&<  (s-r)(2nm^{k-d-1}+1) + (N-s)m^{k-c} \\
&<  N(2n)m^{k-d-1} \\
&\leqslant  \binom{n}{\lfloor n/2 \rfloor}(2n)(n/2)^{k-d} \frac{1}{m}
       \frac{m^{k-d}}{(n/2)^{k-d}} \\
&<  \binom{m-n}{k-d}.
\end{align*}
If $r < \binom{n-1}{d-1}$, then $|\mathcal{F}| < (1+r)\binom{m-n}{k-d}
\leqslant \binom{n-1}{d-1}\binom{m-n}{k-d}$, which is impossible. Hence
$r = \binom{n-1}{d-1}$. Now we show that $s = \binom{n-1}{d-1}$. Note 
that $s \geqslant r = \binom{n-1}{d-1}$. Suppose $s > \binom{n-1}{d-1}$. 
Then by Theorem \ref{DIR}, the image of the injection from $\mathcal{F}_{k-d}$
to ${\sf KG}(n,d) \times {\sf KG}(m-n,k-d)$ cannot be a maximal independent 
set and $|\mathcal{F}_{k-d}| < \binom{n-1}{d-1}\binom{m-n}{k-d}$. This leads 
to $|\mathcal{F}| \leqslant (N-r)\binom{m-n}{k-c} + |\mathcal{F}_{k-d}| < 
\binom{n-1}{c-1}\binom{m-n}{k-c} + \binom{n-1}{d-1}\binom{m-n}{k-d}$,
contradicting the lower bound of $|\mathcal{F}|$ again. Since $r = s = 
\binom{n-1}{d-1}$, we have $|\mathcal{F}| \leqslant \binom{n-1}{d-1}
\binom{m-n}{k-d} + \binom{n-1}{c-1}\binom{m-n}{k-c}$. The equality must 
hold as the right hand side is the known lower bound of $|\mathcal{F}|$.

When $\mathcal{F}$ has maximum cardinality, $\mathcal{F}(A_j) = \{A_j \cup B 
\mid B \in \binom{[n+1,m]}{k-c}\}$ for $j > \binom{n-1}{d-1}$ and
$\mathcal{F}(A_j') = \{A_j' \cup B \mid B \in \binom{[n+1,m]}{k-d}\}$ for
$j \leqslant \binom{n-1}{d-1}$. Now $\{A_j' \mid 1 \leqslant j \leqslant
\binom{n-1}{d-1}\} \subseteq \binom{[n]}{k}$ is a maximum intersecting 
family. Thus, there is a common element $t \in A_j'$ for $1 \leqslant j 
\leqslant \binom{n-1}{d-1}$. On the other hand, no $A_j'$ contains $t$ 
for $j > \binom{n-1}{d-1}$. That implies $t \in A_j$. So $t$ belongs to 
every member of $\mathcal{F}$.
\end{proof}

\begin{theorem}
If integers $n$ and $k$ satisfy $k \leqslant n < 2k-3$, then $\alpha(m,n,k) 
= h(m,n,k)$ holds for sufficiently large $m$. For such a large $m$, 
a maximum $(m,n,k)$-intersecting family is of the form $\mathcal{H}_t$ 
for some $t \in[n]$.
\end{theorem}

\begin{proof}
Let an $(m,n,k)$-intersecting family $\mathcal{F}$ have canonical partition
$\binom{[n]}{k} \cup (\bigcup_{i=1}^{2k-n-1} \mathcal{F}_i)$ as before.
When $n$ is odd, we put $\mathcal{F}_i$ and $\mathcal{F}_{2k-n-i}$ into
a pair for $1 \leqslant i \leqslant (2k-n-1)/2$. When $n$ is even, we put
$\mathcal{F}_i$ and $\mathcal{F}_{2k-n-i}$ into a pair for $1\leqslant i
\leqslant \lfloor (2k-n-1)/2 \rfloor - 1$, and leave $\mathcal{F}_{\lfloor 
(2k-n-1)/2 \rfloor}$ unpaired.

Let $c=k-i$ and $d=n-k+i$. The subfamily $\mathcal{F}_i \cup 
\mathcal{F}_{2k-n-i}$ is an intersecting family and satisfies the conditions 
in Lemma \ref{LEM:FUF}. Therefore $|\mathcal{F}_i| + |\mathcal{F}_{2k-n-i}|
\leqslant \binom{n-1}{k-i-1}\binom{m-n}{i} + \binom{n-1}{n-k+i-1}
\binom{m-n}{2k-n-i}$ for sufficiently large $m$. When $n$ is odd, we 
immediately have the following.
\begin{align*}
|\mathcal{F}| &\leqslant  \binom{n}{k} + \sum_{i=1}^{(2k-n-1)/2}
                            \binom{n-1}{k-i-1}\binom{m-n}{i}+
                            \binom{n-1}{k-i}\binom{m-n}{2k-n-i} \\
              &=  \binom{n}{k} + \sum_{i=1}^{2k-n-1}\binom{n-1}{k-i-1}
                    \binom{m-n}{i}.
\end{align*}
When $n$ is even, we have $|\mathcal{F}_i| \leqslant \binom{n-1}{k-i-1}
\binom{m-n}{i}$ for $i = \lfloor (2k-n-1)/2 \rfloor$ by Theorem \ref{DIR}.
Together with other upper bounds of $|\mathcal{F}_i \cup \mathcal{F}_{2k-n-i}|$, 
we have shown $|\mathcal{F}| \leqslant \binom{n}{k} + \sum_{i=1}^{2k-n-1}
\binom{n-1}{k-i-1}\binom{m-n}{i}$.

When $\mathcal{F}$ is a maximum $(m,n,k)$-intersecting family, for each 
pair $\mathcal{F}_i$ and $\mathcal{F}_{2k-n-i}$, there is an element 
$t_i$ belonging to every member of $\mathcal{F}_i \cup \mathcal{F}_{2k-n-i}$.
This also holds for $\mathcal{F}_i$, $i = \lfloor (2k-n-1)/2 \rfloor$ 
for even $n$. Suppose that there exist $\mathcal{F}_{i_1} \cup
\mathcal{F}_{2k-n-{i_1}}$ and $\mathcal{F}_{i_2} \cup \mathcal{F}_{2k-n-{i_2}}$
for which $t_{i_1} \neq t_{i_2}$. (The case that one of them is $\mathcal{F}_i$, 
$i = \lfloor (2k-n-1)/2 \rfloor$ for even $n$, is the same.) Note that
\[
\mathcal{F}_{2k-n-{i_j}} = \left \{A \cup B \cup \{t_{i_j}\} \mid A \in \binom{[n] 
\setminus \{t_{i_j}\}}{n-k+{i_j}-1}, B \in \binom{[n+1,m]}{2k-n-i_j} 
\right \}
\] 
for $j=1,2$. Since $2(n-k+{i_j}-1) \leqslant n-1$ and $2(2k-n-i_j) < m-n$, 
we can find subsets $F_j \in \mathcal{F}_{2k-n-{i_j}}$ for $j = 1, 2$ such 
that $F_1 \cap F_2 = \emptyset$ if $t_{i_1} \neq t_{i_2}$. Therefore 
$t_{i_1} \neq t_{i_2}$ cannot happen. Consequently, $\mathcal{F} = 
\mathcal{H}_t$ for some $t \in [n]$.
\end{proof}

%
\section{Conclusion}
%

We have introduced the notion of an $(m,n,k)$-intersecting family
and studied its maximum cardinality $\alpha(m,n,k)$. The well-known
theorems of Erd\H{o}s-Ko-Rado and Hilton-Milner in finite extremal
set theory are special cases for $n = k$ and $n = k+1$. The common
cardinality $h(m,n,k)$ of a particular collection of $(m,n,k)$-intersecting
families $\mathcal{H}_t^{m,n,k}$ supplies a natural lower bound for 
$\alpha(m,n,k)$. A noticeable feature of $\mathcal{H}_t^{m,n,k}$ is that 
members of $\mathcal{H}_t^{m,n,k} \setminus \binom{[n]}{k}$ have a nonempty
intersection. We have proved that the families $\mathcal{H}_t^{m,n,k}$ are 
precisely all the $(m,n,k)$-intersecting families of maximum cardinality 
for the cases $n = 2k-1, 2k-3$, and $m$ sufficiently large. When $n = 2k-2$, 
there are other maximum families. Whether $\alpha(m,n,k) = h(m,n,k)$ is true 
in all cases and $\mathcal{H}_t^{m,n,k}$, $n \neq 2k-2$, always characterizes 
maximum families are interesting open problems. Analogue problems can be 
formulated with respect to intersecting families having intersection size 
greater than some prescribed positive integer.

\bigskip
\bigskip

\noindent
{\Large \bf Acknowlegment}. \  This work was done while the first author was a 
post-doctoral fellow in the Institute of Mathematics, Academia Sinica.
The support provided by the Institute is greatly appreciated. 

\bigskip

%

\end{document}